\documentclass{icmart}

\contact[yukinobu.toda@ipmu.jp]{Kavli Institute for the Physics and 
Mathematics of the Universe, University of Tokyo,
5-1-5 Kashiwanoha, Kashiwa, 277-8583, Japan.}

\newtheorem{theorem}{Theorem}[section]
\newtheorem{corollary}[theorem]{Corollary}
\newtheorem{lemma}[theorem]{Lemma}
\newtheorem{proposition}[theorem]{Proposition}
\newtheorem{conjecture}[theorem]{Conjecture}

\theoremstyle{definition}
\newtheorem{definition}[theorem]{Definition}
\newtheorem{remark}[theorem]{Remark}
\newtheorem{exam}[theorem]{Example}

\newcommand{\Hom}{\operatorname{Hom}}

\usepackage[all,ps,dvips]{xy}

\title[Derived category of coherent sheaves and counting invariants]
{Derived category of coherent sheaves and counting invariants}

\author[Yukinobu Toda]
{Yukinobu Toda\thanks{
The author is supported by World Premier 
International Research Center Initiative
(WPI initiative), MEXT, Japan, 
and also by Grant-in Aid
for Scientific Research grant (22684002)
from the Ministry of Education, Culture,
Sports, Science and Technology, Japan.}}

\begin{document}

\begin{abstract}
We survey recent developments on 
Donaldson-Thomas theory, Bridgeland stability conditions and
wall-crossing formula. We emphasize the 
importance of the counting theory of Bridgeland 
semistable objects in the 
derived category of coherent sheaves 
to find a hidden property of the 
generating series of Donaldson-Thomas invariants. 
\end{abstract}



\maketitle

\section{Introduction}
\subsection{Moduli spaces and invariants}
The study of the
\textit{moduli spaces}
is a traditional research subject in algebraic geometry. 
They are schemes or stacks 
whose points bijectively 
correspond to fixed kinds of 
algebro-geometric 
objects, say curves, sheaves on a fixed 
variety, etc. 
These moduli spaces are interesting not only 
in algebraic geometry, but also in connection with 
other research fields such as 
number theory, differential geometry and string theory.
In general it is not easy to study the 
geometric properties of the moduli spaces. 
Instead one tries to 
construct and study the invariants of the moduli spaces, 
e.g. their (weighted) Euler characteristics, 
virtual Poincar\'e or Hodge polynomials, 
integration of the virtual cycles via 
deformation-obstruction theory.
It has been observed that
the best way to study such invariants 
is taking the generating series. 
Sometimes the generating series 
defined from the moduli spaces
have beautiful forms and properties.
Let us observe this phenomenon 
for some rather amenable examples. 
For a quasi-projective variety $X$ (in this article, we always 
assume that the varieties are defined over $\mathbb{C}$), 
the \textit{Hilbert scheme of n-points}
denoted by $\mathrm{Hilb}_n(X)$
is the moduli space of zero dimensional 
subschemes $Z \subset X$
 such that $\chi(\mathcal{O}_Z)=n$.   
 It contains an open subset
corresponding to $n$-distinct points in $X$, 
and the geometric structures of its complement
is in general 
complicated. 
Nevertheless if $X$ is non-singular, 
 the generating series of 
the Euler characteristics of $\mathrm{Hilb}_n(X)$
have the following 
beautiful forms~\cite{Got}, \cite{Cheah}
\begin{align}\label{hilb0}
\sum_{n\ge 0} \chi(\mathrm{Hilb}_n(X))q^n 
=\left\{  \begin{array}{ll}
(1-q)^{-\chi(X)}, & \dim X=1 \\
\prod_{m\ge 1}(1-q^m)^{-\chi(X)}, & \dim X=2 \\
\prod_{m\ge 1}(1-q^m)^{-m\chi(X)}, & \dim X=3. 
\end{array}
\right.
\end{align}
In the case 
$X=\mathbb{C}^d$, the
torus localization shows that 
$\chi(\mathrm{Hilb}_n(\mathbb{C}^d))$ coincides with the 
number of $d$-dimensional partitions of $n$. 
The resulting product formulas are consequences of 
enumerative combinatorics, as in~\cite{Stan}. 
A general case is reduced to the case $X=\mathbb{C}^d$. 

\subsection{Curve counting invariants}
The study of the invariants of the moduli spaces of curves inside a 
variety is more important and interesting, because of 
its connection with \textit{world sheet counting}
in string theory. 
A particularly important case is 
when $X$ is a \textit{Calabi-Yau 3-fold}, i.e. 
$X$ has a trivial canonical line bundle with 
$H^1(X, \mathcal{O}_X)=0$, as 
the string theory predicts our universe to be 
the product of the four dimensional space time 
with a Calabi-Yau 3-fold. 
Similarly to $\mathrm{Hilb}_n(X)$, we
denote by $\mathrm{Hilb}_n(X, \beta)$
the \textit{Hilbert scheme of curves}
inside $X$, that is the moduli space of projective subschemes 
$C \subset X$ with $\dim C \le 1$, 
$[C]=\beta$ and $\chi(\mathcal{O}_{C})=n$. 
The following result was obtained by the author in 2008, 
and plays a key role in this article: 
\begin{theorem}[\cite{Tolim2}, \cite{Tcurve1}]\label{intro:thm}
Let $X$ be a smooth projective Calabi-Yau 3-fold. 
Then for fixed $\beta \in H_2(X, \mathbb{Z})$, 
the quotient series
\begin{align}\label{intro:hilb}
\frac{\sum_{n\in \mathbb{Z}} \chi(\mathrm{Hilb}_n(X, \beta))q^n}{\sum_{n\ge 0} \chi(\mathrm{Hilb}_n(X))q^n}
\end{align}
is the Laurent expansion of a rational function of $q$, invariant 
under $q\leftrightarrow 1/q$. 
\end{theorem}
Note that the denominator of (\ref{intro:hilb}) 
is given by the formula (\ref{hilb0}). 
A typical example of a rational function of $q$ invariant under
$q \leftrightarrow 1/q$ is 
$q/(1+q)^2 =q-2q^2 + 3q^3 - \cdots$. 
We remark that the invariance of $q \leftrightarrow 1/q$ does 
not say the invariance of the generating series 
after the formal substitution $q \mapsto 1/q$, 
but so after taking the analytic continuation of the function 
(\ref{intro:hilb}) from $\lvert q \rvert \ll 1$ to 
$\lvert q \rvert \gg 1$. 
The above result was conjectured in~\cite{Li-Qin}
as the \textit{unweighted version}
of the rationality conjecture of rank one 
Donaldson-Thomas (DT) invariants by 
Maulik-Nekrasov-Okounkov-Pandharipande (MNOP)~\cite{MNOP}.
The rationality conjecture 
was proposed in order to 
formulate the \textit{Gromov-Witten/Donaldson-Thomas correspondence conjecture}
comparing two kinds of curve counting invariants on Calabi-Yau 3-folds. 

The DT invariant was introduced by Thomas~\cite{Thom}, 
as a holomorphic analogue of Casson invariants of 
real 3-manifolds.
It counts stable coherent sheaves on a Calabi-Yau 3-fold, 
 and is a higher dimensional generalization of 
Donaldson invariants on algebraic surfaces. 
For a Calabi-Yau 3-fold $X$, an ample divisor $H$ on $X$ 
and a cohomology class $v\in H^{\ast}(X, \mathbb{Q})$, 
the DT invariant $\mathrm{DT}_H(v) \in \mathbb{Z}$
is defined to be the degree of the zero 
dimensional virtual fundamental cycle on the 
moduli space of $H$-stable 
coherent sheaves $E$ on $X$ with $\mathrm{ch}(E)=v$. 
It also coincides with the weighted Euler characteristic 
with respect to the Behrend's constructible function
on that moduli space~\cite{Beh}. 
The DT invariants were later
generalized by Joyce-Song~\cite{JS} and Kontsevich-Soibelman~\cite{K-S}
so that they also count strictly $H$-semistable sheaves.
The generalized DT invariants involve
the Behrend functions and the motivic 
Hall algebras in the definition, 
and they are $\mathbb{Q}$-valued. \\
\quad  
The Hilbert scheme of points or curves on a Calabi-Yau 3-fold
is also interpreted as a moduli space of stable sheaves, 
by assigning a subscheme $C \subset X$ with 
its ideal sheaf $I_C \subset \mathcal{O}_X$. 
The resulting DT invariant is the weighted Euler characteristic 
of the Hilbert scheme of points or curves, and in particular it is independent 
of $H$. In this sense, the invariant $\chi(\mathrm{Hilb}_n(X, \beta))$ is 
the unweighted version of the DT invariant, 
which coincides with the honest DT invariant up to sign if 
$\mathrm{Hilb}_n(X, \beta)$ 
is non-singular.  
The result 
of Theorem~\ref{intro:thm}
for the weighted version
was later proved by Bridgeland~\cite{BrH}. 
The rationality property and the invariance of
$q \leftrightarrow 1/q$ of the series (\ref{intro:hilb})
are not visible if we just look at the moduli spaces of 
curves or points. 
Such hidden properties of the series (\ref{intro:hilb}) 
are visible after we develop new moduli theory and invariants 
of objects in the derived category of coherent sheaves. 

\subsection{Derived category of coherent sheaves}
Recall that for a variety $X$, the bounded derived category of 
coherent sheaves $D^b \mathrm{Coh}(X)$
is defined to be the localization by  
quasi-isomorphisms of the homotopy category of the bounded
complexes of coherent sheaves on $X$. 
The derived category is no longer an abelian category, 
but has a structure of a triangulated category. 
It was originally introduced by Grothendieck in 1960's 
in order to formulate the relative version of 
Serre duality theorem, known as Grothendieck duality theorem. 
Later it was observed by Mukai~\cite{Mu1} that 
an abelian variety and its dual abelian variety, which 
are not necessary isomorphic in general, have
the equivalent derived categories of coherent sheaves. 
This phenomena suggests that 
the category $D^b \mathrm{Coh}(X)$ has more symmetries than 
the category of coherent sheaves, as the 
latter category is known to reconstruct the original variety. 
Such a phenomena has drawn much attention since 
Kontsevich
proposed the \textit{Homological mirror symmetry conjecture}
in~\cite{Kon}. It 
predicts an equivalence between the derived category of 
coherent sheaves on a Calabi-Yau manifold and 
the derived Fukaya category of its mirror manifold, 
based on an insight that the derived category 
$D^b \mathrm{Coh}(X)$ is a mathematical 
framework of D-branes of type B in string theory. 
There have been several developments in constructing
Mukai type derived equivalences between 
non-isomorphic varieties~\cite{Br1}, \cite{Or1}, \cite{BorCal}, \cite{Ka1}, 
and non-trivial autequivalences~\cite{ST}, \cite{HT}, 
based on the ideas from mirror symmetry. 
Furthermore such Mukai type equivalences have been discovered beyond 
algebraic geometry. 
For instance, derived McKay correspondence~\cite{BKR}
gives an equivalence between the derived 
category of finite group representations and the 
derived category of coherent sheaves 
on the crepant resolution of the quotient singularity. 
This is now interpreted as a special 
case of equivalences between 
usual commutative varieties and 
non-commutative varieties
in the context of Van den Bergh's non-commutative 
crepant resolutions~\cite{MVB}. 
There also exists an Orlov's equivalence~\cite{Orsin}
between the derived category of coherent 
sheaves on a Calabi-Yau hypersurface in the projective
space and the category of graded matrix factorizations of 
the defining equation of it. 
This result, called Landau-Ginzbrug/Calabi-Yau correspondence, 
was also motivated by mirror symmetry. 
Now it is understood that the derived 
categories have more symmetries than 
the categories of coherent sheaves. 
Our point of view is to make the
hidden properties of the generating 
series of DT type invariants visible via 
symmetries in the derived categories. 

\subsection{Bridgeland stability conditions}
The idea of applying derived categories to 
the study of generating series of DT type invariants 
suggests an importance of 
constructing 
moduli spaces and invariants of 
objects in the derived categories. 
Note that in constructing the original 
DT invariants, we need to fix an ample 
divisor on a Calabi-Yau 3-fold $X$, 
and the associated stability 
condition on $\mathrm{Coh}(X)$
in order to construct a good moduli 
space of stable sheaves. The notion 
of stability conditions on triangulated categories, 
in particular on derived categories of coherent sheaves, 
was introduced by Bridgeland~\cite{Brs1} 
as a mathematical framework of Douglas's $\Pi$-stability~\cite{Dou2}
in string theory. 
For a triangulated category $\mathcal{D}$, 
a Bridgeland stability condition on it
roughly consists of data 
$\sigma=(Z, \{\mathcal{P}(\phi)\}_{\phi \in \mathbb{R}})$ 
for a group 
homomorphism $Z \colon K(\mathcal{D}) \to \mathbb{C}$
called the \textit{central charge}, 
and the collection of subcategories $\mathcal{P}(\phi) \subset \mathcal{D}$
for $\phi \in \mathbb{R}$
whose objects are called
 $\sigma$\textit{-semistable objects with phase} $\phi$.
The main result by Bridgeland~\cite{Brs1}
is to show that the set of `good' stability 
conditions on $\mathcal{D}$ forms a complex manifold. 
This complex manifold is in particular important 
when $\mathcal{D}=D^b \mathrm{Coh}(X)$ for a Calabi-Yau manifold $X$. 
In this case, the space of 
stability conditions $\mathrm{Stab}(X)$ is 
expected to contain the universal covering space 
of the moduli space of complex structures of 
a mirror manifold of $X$. 
So far the space $\mathrm{Stab}(X)$ has been 
studied in several situations, e.g. 
$X$ is a curve~\cite{Brs1}, \cite{Mac2},
 $X$ is a K3 surface~\cite{Brs2}, 
$X$ is a some non-compact Calabi-Yau 
3-fold~\cite{Brs4},~\cite{BaMa},~\cite{Tst},~\cite{Tst2}.
On the other hand, there has been a serious issue in 
studying Bridgeland stability conditions
on projective Calabi-Yau 3-folds which are likely to be the most 
important case: 
we are not able to prove the existence of Bridgeland 
stability conditions on smooth projective Calabi-Yau 3-folds.
In~\cite{BMT}, the existence problem is reduced to showing 
a conjectural Bogomolov-Gieseker type inequality evaluating the third
Chern character of certain two term complexes of coherent sheaves. 
However proving that inequality conjecture seems to 
require a new idea.  

\subsection{New invariants via derived categories}\label{subsec:new}
Let $X$ be a smooth projective Calabi-Yau 3-fold. 
We expect that, for a given $\sigma \in \mathrm{Stab}(X)$
and $v\in H^{\ast}(X, \mathbb{Q})$, there 
exists the DT type invariant $\mathrm{DT}_{\sigma}(v) \in \mathbb{Q}$
which counts $\sigma$-semistable objects $E\in D^b \mathrm{Coh}(X)$
with $\mathrm{ch}(E)=v$. As we mentioned, there is 
a serious issue in constructing a Bridgeland stability condition on 
projective Calabi-Yau 3-folds, but let us ignore this for a while. 
For an ample divisor $H$ on $X$, we expect that 
the classical $H$-stability appears as a certain special 
limiting point in $\mathrm{Stab}(X)$ called the \textit{large volume limit}. 
If we take $\sigma \in \mathrm{Stab}(X)$ near the large volume limit 
point, then we expect the equality $\mathrm{DT}_{\sigma}(v)=\mathrm{DT}_H(v)$. 
On the other hand, suppose that 
there is an autequivalence $\Phi$ of $D^b \mathrm{Coh}(X)$
and $\tau \in \mathrm{Stab}(X)$
so that the equality 
$\mathrm{DT}_{\tau}(v)=\mathrm{DT}_{\tau}(\Phi_{\ast}v)$
holds for any $v$. Then the generating series of the 
invariants $\mathrm{DT}_{\tau}(v)$
is preserved by the variable change induced by 
$v\mapsto \Phi_{\ast}v$. 
If we are able to relate 
$\mathrm{DT}_{\sigma}(v)$ and $\mathrm{DT}_{\tau}(v)$, 
then it would imply the hidden symmetry 
of the generating series of classical DT invariants $\mathrm{DT}_{H}(v)$
with respect $v\mapsto \Phi_{\ast}v$. 
The relationship between $\mathrm{DT}_{\sigma}(v)$ and $\mathrm{DT}_{\tau}(v)$
is studied by the wall-crossing phenomena:
there should be a wall and chamber structure on the 
space $\mathrm{Stab}(X)$ so that the invariants 
$\mathrm{DT}_{\ast}(v)$ are constant on a chamber but jumps 
if $\ast$ crosses a wall. 
The wall-crossing formula of the invariants $\mathrm{DT}_{\ast}(v)$
should be described by a general framework 
established by Joyce-Song~\cite{JS}, Kontsevich-Soibelman~\cite{K-S}, 
using stack theoretic Hall algebras. \\
\quad 
However as we mentioned, we are not able to prove 
$\mathrm{Stab}(X) \neq \emptyset$, so the above 
story is the next stage after proving the non-emptiness. 
The idea of proving Theorem~\ref{intro:thm} 
was to introduce `weak' Bridgeland stability conditions
on triangulated categories, 
and 
apply the above story for the space of weak stability conditions
on the subcategory of $D^b \mathrm{Coh}(X)$
generated by $\mathcal{O}_X$ and one or zero dimensional sheaves. 
The latter subcategory is 
called the category of \textit{D0-D2-D6 bound states}. 
The notion of weak stability conditions is a kind of 
limiting degenerations of Bridgeland stability conditions, 
and it is
 a coarse version of Bayer's polynomial stability conditions~\cite{Bay}, 
the author's limit stability conditions~\cite{Tolim}. 
It is easier to construct weak stability conditions and 
enough to prove Theorem~\ref{intro:thm} applying the above 
story. The 
derived dual $E \mapsto \mathbf{R} \mathcal{H} om(E, \mathcal{O}_X)$, 
an autequivalence of $D^b \mathrm{Coh}(X)$, 
turned out to be responsible for the 
hidden symmetric property of $q \leftrightarrow 1/q$
of the series (\ref{intro:hilb}) in the above story. \\
\quad 
The idea of proving Theorem~\ref{intro:thm} has turned out to be 
useful in proving several other interesting properties of 
DT type invariants, say DT/PT correspondence~\cite{Tcurve1}, \cite{BrH}
conjectured by Pandharipande-Thomas~\cite{PT}. 
We refer to~\cite{Tcurve2}, \cite{Cala}, \cite{ToBPS}, \cite{Tcurve3}, \cite{Trk2}, \cite{Stop}, \cite{Nhig}, \cite{Tst3}, \cite{TodK3}, \cite{TodBG}
for other works relating the above story.

\subsection{Plan of this article}
In Section~\ref{sec:DT}, we review and 
survey recent developments of Donaldson-Thomas theory. 
In Section~\ref{sec:Br}, we survey the developments on 
Bridgeland stability conditions. 
In Section~\ref{sec:Con}, 
we discuss open problems on DT theory and 
Bridgeland stability conditions.

\section{Donaldson-Thomas theory}\label{sec:DT}
\subsection{Moduli spaces of semistable sheaves}
Let $X$ be a smooth projective variety and $H$ an ample 
divisor on $X$. For an object $E \in \mathrm{Coh}(X)$, 
its Hilbert polynomial is given by 
\begin{align*}
\chi(E \otimes \mathcal{O}_{X}(mH))=
a_d m^d + a_{d-1} m^{d-1} + \cdots
\end{align*}
for $a_i \in \mathbb{Q}$
by the Riemann-Roch theorem. 
Here $a_d \neq 0$ and $d$ is the dimension of the support of $E$. 
The reduced Hilbert polynomial $\overline{\chi}_H(E, m)$
is defined to be $\chi(E\otimes \mathcal{O}_X(mH))/a_d$. 
\begin{definition}
An object $E \in \mathrm{Coh}(X)$ is $H$\textit{-(semi)stable}
if for any subsheaf $0\neq F \subsetneq E$, 
we have $\dim \mathrm{Supp}(F)= \dim \mathrm{Supp}(E)$ and 
the inequality 
$\overline{\chi}_H(F, m) <(\le) \overline{\chi}_H(E, m)$
holds 
for $m\gg 0$. 
\end{definition}
\begin{remark}
Note that if $E$ is torsion free, 
then $\overline{\chi}_H(F, m)=m^d + c \cdot \mu_H(E) m^{d-1} + O(m^{d-2})$
where $d=\dim X$, 
$\mu_H(E)=c_1(E)H^{d-1}/\mathrm{rank}(E)$, and $c$ is some constant. 
Hence the $H$-(semi)stability is the refinement 
of $H$-slope (semi)stability defined by the slope function
$\mu_H(\ast)$. 
\end{remark}
Let $\mathcal{C}oh(X)$ be the 2-functor from 
the category of complex schemes to the groupoid, whose
$S$-valued points form the groupoid 
of flat families of coherent sheaves on $X$ over $S$. 
The 2-functor $\mathcal{C}oh(X)$ forms 
a stack, which is known to be an Artin stack
locally of finite type, but neither finite type nor 
separated. 
The situation becomes better if we consider the substacks
for $v \in H^{\ast}(X, \mathbb{Q})$
\begin{align*}
\mathcal{M}_H^{\rm{s}}(v) \subset \mathcal{M}_H^{\rm{ss}}(v) \subset
\mathcal{C}oh(X).
\end{align*}
Here $\mathcal{M}_H^{\rm{s(ss)}}(v)$ is the substack of $H$-(semi)stable
$E \in \mathrm{Coh}(X)$ with $\mathrm{ch}(E)=v$, which 
is an open substack of $\mathcal{C}oh(X)$. 
The stack $\mathcal{M}_H^{\rm{ss}}(v)$ is 
of finite type but not separated in general. 
The stack $\mathcal{M}_H^{\rm{s}}(v)$ is 
of finite type, separated, and 
a $\mathbb{C}^{\ast}$-gerb over a 
quasi-projective scheme $M_H^{\rm{s}}(v)$. 
The scheme $M_H^{\rm{s}}(v)$
is projective if 
$\mathcal{M}_H^{\rm{s}}(v)=\mathcal{M}_H^{\rm{ss}}(v)$.

\subsection{Donaldson-Thomas invariants}
Let $X$ be a smooth projective 3-fold. 
We say it is a \textit{Calabi-Yau 3-fold} if 
$K_X=0$ and $H^1(X, \mathcal{O}_X)=0$. 
A typical example is a quintic hypersurface in $\mathbb{P}^4$. 
Let $H$ be an ample divisor on $X$, $v$
an element in $H^{\ast}(X, \mathbb{Q})$, 
and consider the moduli scheme $M_H^{\rm{s}}(v)$. 
A standard deformation 
theory of sheaves (cf.~\cite{Hu})
shows that the tangent space at $[E] \in M_H^{\rm{s}}(v)$
is given by $\mathrm{Ext}^1(E, E)$, 
and the obstruction space is given by 
$\mathrm{Ext}^2(E, E)$. 
The Calabi-Yau condition and the Serre duality 
implies that the latter space is dual to 
$\mathrm{Ext}^1(E, E)$. Hence the 
virtual dimension at $[E]$, defined to 
be the dimension of the tangent space minus 
the dimension of the obstruction space, is 
zero which is independent of $E$. 
Based on this observation, Thomas~\cite{Thom}
constructed two term complex of vector bundles $\mathcal{E}^{\bullet}$
on 
$M_H^{\rm{s}}(v)$ and a 
morphism $\mathcal{E}^{\bullet} \to L_{M_H^{\rm{s}}(v)}$
in $D(M_H^{\rm{s}}(v))$, giving 
a symmetric perfect obstruction theory 
in the sense of Behrend-Fantechi~\cite{BF},~\cite{BBr}. 
By the construction in~\cite{BF},  
there is the associated zero dimensional 
virtual cycle $[M_H^{\rm{s}}(v)]^{\rm{vir}}$
on $M_H^{\rm{s}}(v)$, and
we are able to take its degree if $M_H^{\rm{s}}(v)$ is projective. 
The DT invariant is defined as follows: 
\begin{definition}
If $\mathcal{M}_H^{\rm{s}}(v)=\mathcal{M}_H^{\rm{ss}}(v)$ holds, 
we define $\mathrm{DT}_H(v) \in \mathbb{Z}$ to be
the degree of $[M_H^{\rm{s}}(v)]^{\rm{vir}}$. 
\end{definition}
The above construction via the virtual 
cycle easily shows that the DT invariant is invariant
under deformations of complex structures of $X$. 
However in practice, it is more 
convenient to describe the DT invariant 
in terms of 
Behrend's constructible function~\cite{Beh}. 
The Behrend function is easily described if we
use the following 
result by Joyce-Song~\cite{JS}: 
\begin{theorem}[\cite{JS}]\label{thm:CS}
for any $p \in M_H^{\rm{s}}(v)$, 
there is an analytic open subset $p \in U \subset M_H^{\rm{s}}(v)$, 
a complex manifold $V$ and a holomorphic function 
$f \colon V \to \mathbb{C}$ such that 
$U$ is isomorphic to $\{df=0\}$. 
\end{theorem}
Using the above result, the 
Behrend function $\nu_B$ on $M_H^{\rm{s}}(v)$ is 
described as
\begin{align*}
\nu_B(p)= (-1)^{\dim V}(1-\chi(M_p(f)))
\end{align*}
where $M_p(f)$ is the Milnor fiber of $f$ at $p$. 
The function $\nu_B$ is shown to a be well-defined 
constructible function on $M_H^{\rm{s}}(v)$. 
\begin{theorem}[\cite{Beh}]
If $\mathcal{M}_H^{\rm{s}}(v)=\mathcal{M}_H^{\rm{ss}}(v)$ holds, 
we have the equality
\begin{align}\label{DT=B}
\mathrm{DT}_H(v)=\sum_{k\in \mathbb{Z}}
k \cdot \chi(\nu_B^{-1}(k)). 
\end{align}
In particular if $M_H^{\rm{s}}(v)$ is non-singular and connected, 
the invariant $\mathrm{DT}_H(v)$ coincides with 
$\chi(M_H^{\rm{s}}(v))$ up to sign. 
\end{theorem}
Based on the above description of the DT invariant, 
Joyce-Song~\cite{JS}
and Kontsevich-Soibelman~\cite{K-S}
constructed the generalized DT invariant
$\mathrm{DT}_H(v) \in \mathbb{Q}$
without the condition $\mathcal{M}_H^{\rm{s}}(v)=\mathcal{M}_H^{\rm{ss}}(v)$. 
The construction uses the stack theoretic 
Hall algebra $H(\mathrm{Coh}(X))$ of $\mathrm{Coh}(X)$, 
and its well-definedness is highly non-trivial.
A very rough description of it may be
\begin{align*}
\mathrm{DT}_H(v) = \int_{\log \mathcal{M}_H^{\rm{ss}}(v)} 
\nu_B \cdot d\chi. 
\end{align*} 
The `log' is taken in the 
algebra $H(\mathrm{Coh}(X))$. 
Some more explanation of a specific case 
is available in~\cite{Tsurvey}.

\begin{remark}
We can define another invariant 
$\mathrm{DT}_H^{\chi}(v) \in \mathbb{Q}$
by formally putting $\nu_{B} \equiv 1$ in the 
definition of $\mathrm{DT}_H(v)$. 
If $\mathcal{M}_H^{\rm{s}}(v)=\mathcal{M}_H^{\rm{ss}}(v)$, 
it coincides with the usual 
Euler characteristic $\chi(M_H^{\rm{s}}(v))$.
When we say a result as a
\textit{weighted (resp. an unweighted) version}, 
it means the result for the invariants 
$\mathrm{DT}_H(v)$ (resp. $\mathrm{DT}_H^{\chi}(v)$). 
\end{remark}

\subsection{Rank one DT invariants}
In what follows, we identify 
$H^4(X, \mathbb{Q})$, $H^6(X, \mathbb{Q})$ with $H_2(X, \mathbb{Q})$, 
$\mathbb{Q}$ respectively 
by the Poincar\'e duality. 
Given $\beta \in H_2(X, \mathbb{Z})$ and $n\in \mathbb{Z}$, 
it is easy to show that
$\mathrm{Hilb}_{n}(X, \beta)$
is isomorphic to $M_H^{\rm{s}}(v)$
for $v=(1, 0, -\beta, -n)$
by the assignment $C \mapsto I_C$. 
The resulting invariant
\begin{align*}
I_{n, \beta}=\mathrm{DT}_H(1, 0, -\beta, -n) \in \mathbb{Z}
\end{align*}
is independent of $H$, and it counts
one or zero dimensional subschemes $C \subset X$
with $[C]=\beta$, $\chi(\mathcal{O}_C)=n$. 
For $\beta \in H_2(X, \mathbb{Z})$, 
the series $I_{\beta}(X)$ is defined to be
\begin{align*}
I_{\beta}(X)=\sum_{n \in \mathbb{Z}}
I_{n, \beta}q^n. 
\end{align*}
\begin{exam}\label{exii}
(i) 
If $\beta=0$, we have~\cite{Li}, \cite{BBr}, \cite{LP}
\begin{align*}
I_0(X)=
\prod_{k \ge 1} (1-(-q)^k)^{-k \chi(X)}. 
\end{align*}

(ii) If $f \colon X \to Y$ is a birational contraction 
whose exceptional locus is $C\cong \mathbb{P}^1$
with normal bundle $\mathcal{O}_C(-1)^{\oplus 2}$, we have~\cite{BeBryan}
\begin{align*}
\sum_{m\ge 0} I_{m[C]}(X)t^m=
\prod_{k \ge 1} (1-(-q)^k)^{-k \chi(X)}
\prod_{k\ge 1}(1-(-q)^k t)^k. 
\end{align*}
\end{exam} 
The above example indicates that
the quotient series $I_{\beta}(X)/I_{0}(X)$
is the honest curve counting series with homology class
$\beta$. The following conjecture was proposed by 
MNOP~\cite{MNOP}:
\begin{conjecture}[\cite{MNOP}]\label{MNOP}
(i) The quotient series 
$I_{\beta}(X)/I_0(X)$ is the Laurent expansion of a 
rational function of $q$, invariant under $q\leftrightarrow 1/q$. 

(ii) 
After the variable change $q=-e^{i\lambda}$, we have 
the equality
\begin{align*}
\sum_{\beta \ge 0} \frac{I_{\beta}(X)}{I_0(X)}t^{\beta}
= \exp\left( \sum_{g\ge 0, \beta>0}
\mathrm{GW}_{g, \beta}(X)\lambda^{2g-2} t^{\beta}
\right). 
\end{align*}
\end{conjecture}
Here $\mathrm{GW}_{g, \beta}(X)\in \mathbb{Q}$ is the Gromov-Witten invariant
counting stable maps $f \colon C \to X$
from projective curves $C$ with at worst nodal singularities
with $g(C)=g$, $f_{\ast}[C]=\beta$. 
The variable change $q=-e^{i\lambda}$ makes sense by 
the rationality conjecture (i). 
The above conjecture was first proved for 
toric Calabi-Yau 3-folds in~\cite{MNOP}. 

\subsection{Developments on MNOP conjecture}
As we mentioned in the introduction, 
the result of Theorem~\ref{intro:thm}
is the unweighted version of Conjecture~\ref{MNOP} (i). 
The weighted version was  
proved in~\cite{BrH}. 
We have the following result
~\cite{Tolim2}, \cite{Tcurve1} (unweighted
version), \cite{BrH} (weighted version):
\begin{theorem}\label{thm:main}
There exist invariants $N_{n, \beta} \in \mathbb{Q}$, 
$L_{n, \beta} \in \mathbb{Q}$
satisfying 
\begin{itemize}
\item 
$N_{n, \beta}=N_{-n, \beta}=N_{n+H \beta, \beta}$
for any ample divisor $H$ on $X$,
\item 
$L_{n, \beta}=L_{-n, \beta}$, and it is 
zero for $\lvert n \rvert \gg 0$,
\end{itemize}
such that we have the following formula:
\begin{align*}
\sum_{\beta \ge 0}I_{\beta}(X)t^{\beta}=
\prod_{n>0, \beta \ge 0}
\exp((-1)^{n-1}nN_{n, \beta}q^n t^{\beta})
\left( \sum_{n, \beta}L_{n, \beta} q^n t^{\beta} \right). 
\end{align*}
\end{theorem}
\begin{remark}
The proofs for the unweighted version
in the author's papers \cite{Tcurve1}, \cite{Tolim2}
can be modified to show the weighted version, 
if once a similar result of Theorem~\ref{thm:CS}
for the moduli spaces of complexes in~\cite{Inaba}, \cite{LIE}
is shown to be true (cf.~\cite{Tsurvey}). 
This is also applied for the results below. 
\end{remark}
The rationality conjecture is an easy consequence of Theorem~\ref{thm:main}:
\begin{corollary}
Conjecture~\ref{MNOP} (i) is true. 
\end{corollary}
There exist
geometric meanings of $N_{n, \beta}$ and $L_{n, \beta}$. 
The former invariant 
is nothing but the generalized DT invariant 
$\mathrm{DT}_{H}(0, 0, \beta, n)$, 
which counts one or zero dimensional 
$H$-semistable sheaves $F$ on $X$
with $[F]=\beta$, $\chi(F)=n$. 
A priori, $N_{n, \beta}$
 is defined using the ample divisor $H$, 
but the resulting invariant 
is shown to be
independent of $H$. 
The latter invariant $L_{n, \beta}$ is 
more interesting. It counts 
certain two term complexes $E \in  D^b \mathrm{Coh}(X)$
(indeed they are perverse coherent sheaves in the 
sense of~\cite{Kashi}, \cite{Bez})
satisfying $\mathrm{ch}(E)=(1, 0, -\beta, -n)$, 
which are semistable with respect to 
a derived self dual weak stability condition on it. 
The result of Theorem~\ref{thm:main} is proved 
along with the idea stated in Subsection~\ref{subsec:new}. 

A similar idea also proves Pandharipande-Thomas
conjecture~\cite{PT}
relating the quotient series of 
rank one DT invariants with the invariants 
counting stable pairs. 
The definition of stable pairs is given as follows: 
\begin{definition}[\cite{PT}]
A \textit{stable pair} is data $(F, s)$ where $F$ is a pure 
one dimensional sheaf on $X$, $s \colon \mathcal{O}_X \to F$
is a morphism which is surjective in dimension one. 
\end{definition}
A typical example of a stable pair is $(\mathcal{O}_C(D), s)$, 
where $C \subset X$ is a smooth curve, $D \subset C$
is an effective divisor and $s$ is a natural 
composition $\mathcal{O}_X \twoheadrightarrow \mathcal{O}_C \subset
\mathcal{O}_C(D)$. 
For given $\beta \in H_2(X, \mathbb{Z})$ and $n \in \mathbb{Z}$, 
the moduli space $P_n(X, \beta)$
of stable pairs $(F, s)$ with $[F]=\beta$, $\chi(F)=n$
is a projective scheme with a symmetric perfect obstruction 
theory~\cite{PT}. 
The PT invariant $P_{n, \beta} \in \mathbb{Z}$
is defined to be the degree of the 
zero dimensional virtual fundamental cycle $[P_n(X, \beta)]^{\rm{vir}}$
on $P_n(X, \beta)$. 
The invariant $P_{n, \beta}$
is deformation invariant, and coincides with the 
weighted Euler characteristic with respect to the 
Behrend function on $P_n(X, \beta)$. 
The following conjecture was proposed by~\cite{PT}, 
its unweighted version was proved in~\cite{Tcurve1}, \cite{StTh}, 
and the weighted version was proved in~\cite{BrH}:
\begin{theorem}\label{thm:DTPT}
For fixed $\beta\in H_2(X, \mathbb{Z})$, 
we have the equality of the generating series
\begin{align*}
\frac{I_{\beta}(X)}{I_0(X)}
=\sum_{n\in \mathbb{Z}} P_{n, \beta}q^n.
\end{align*}
\end{theorem}
Finally in~\cite{PP}, Pandharipande-Pixton
 proved Conjecture~\ref{MNOP} (ii) 
for large class of Calabi-Yau 3-folds including 
quintic hypersurfaces in $\mathbb{P}^4$:
\begin{theorem}[\cite{PP}]\label{thm:PP}
Conjecture~\ref{MNOP} (ii) is true 
if $X$ is a complete intersection Calabi-Yau 3-fold 
in the product of projective spaces. 
\end{theorem}
Indeed what they proved is the correspondence
between Gromov-Witten invariants and 
stable pair invariants. 
Combined with Theorem~\ref{thm:DTPT}, 
the result of Theorem~\ref{thm:PP} was proved. 
Their proof relies on the degeneration formula of 
GW and PT invariants, and the torus localization formula. 

\subsection{Non-commutative DT theory and flops}\label{subsec:ncDT}
The DT theory can be also constructed for non-commutative 
varieties or algebras.
Let $Y$ be a quasi-projective 3-fold which admits
two crepant small resolutions giving a \textit{flop}:
\begin{align}\label{small}
\phi \colon X \stackrel{f}{\to}Y \stackrel{f^{\dag}}{\leftarrow}X^{\dag}.
\end{align}
In this situation, Van den Bergh~\cite{MVB}
constructed sheaf of non-commutative algebras 
$A_Y$ on $Y$ and derived equivalences
\begin{align}\label{genMC}
D^b \mathrm{Coh}(X^{\dag}) \stackrel{\Psi}{\to}
D^b \mathrm{Coh}(A_Y) \stackrel{\Phi}{\to}
D^b \mathrm{Coh}(X)
\end{align}
so that their composition gives Bridgeland's
flop equivalence~\cite{Brs1}. For $n\in \mathbb{Z}$
and $\beta \in H_2(X, \mathbb{Z})$, 
let $\mathrm{Hilb}_n(A_Y, \beta)$ be the moduli space of 
surjections $A_Y \twoheadrightarrow F$
in $\mathrm{Coh}(A_Y)$ 
such that $F$ has at most one dimensional support
and $[\Phi(F)]=\beta$, $\chi(\Phi(F))=n$. 
If $X$ is a smooth projective Calabi-Yau 3-fold, 
there is a symmetric perfect obstruction theory 
on $\mathrm{Hilb}_n(A_Y, \beta)$, 
and the degree of its zero dimensional 
virtual fundamental cycle 
defines the 
\textit{non-commutative DT (ncDT) invariant}
$A_{n, \beta} \in \mathbb{Z}$. 
Alternatively, $A_{n, \beta}$ is 
defined to be the weighted Euler characteristic 
of the Behrend function on $\mathrm{Hilb}_n(A_Y, \beta)$. 
We set $I_{\beta}(A_Y)$ to be
\begin{align*}
I_{\beta}(A_Y)=\sum_{n\in \mathbb{Z}}A_{n, \beta}q^n. 
\end{align*}
The following result was proved in~\cite{Tcurve2} for 
the unweighted
version, and~\cite{Cala} for the weighted version, 
basically along with the argument in Subsection~\ref{subsec:new}:
\begin{theorem}\label{thm:flop}
We have the following identities:
\begin{align*}
&\sum_{f_{\ast}\beta=0}I_{\beta}(A_Y)t^{\beta} =
\prod_{k\ge 1}(1-(-q)^k)^{k\chi(X)}\left(\sum_{f_{\ast}\beta=0} I_{\beta}(X)t^{\beta} \right)
\left(\sum_{f_{\ast}\beta=0} I_{-\beta}(X)t^{\beta} \right) \\
&\frac{\sum_{\beta}I_{\beta}(X) t^{\beta}}{\sum_{f_{\ast}\beta=0}
I_{\beta}(X) t^{\beta}}
=\frac{\sum_{\beta}I_{\beta}(A_Y) t^{\beta}}{\sum_{f_{\ast}\beta=0}
I_{\beta}(A_Y) t^{\beta}}
=\frac{\sum_{\beta}I_{\phi_{\ast}\beta}(X^{\dag}) 
t^{\beta}}{\sum_{f_{\ast}\beta=0}
I_{\phi_{\ast}\beta}(X^{\dag}) t^{\beta}}. 
\end{align*}
\end{theorem}

\begin{exam}
Let 
$Y=(xy+zw=0) \subset \mathbb{C}^4$
be the conifold singularity, and 
take two crepant small resolutions (\ref{small})
by blowing up at the ideals 
$(x, z)$ and $(x, w)$. 
In this case, the algebra $A_Y$
is the path algebra of the following quiver
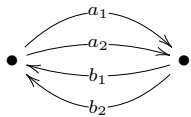
\begin{align*}
\xymatrix{
\bullet
 \ar@/^/[rr]|{a_2} \ar@/^1.5pc/[rr]|{a_1}&&
\bullet \ar@/^/[ll]|{b_1} \ar@/^1.5pc/[ll]|{b_2} 
} 
\end{align*}
with relation given by the derivations of the 
super potential $W=a_1 b_1 a_2 b_2-a_1 b_2 a_2 b_1$. 
Although $X$ is not projective in this case, 
the ncDT invariant $A_{n, m[C]} \in \mathbb{Z}$
makes sense, and coincides with the weighted Euler characteristic 
of the moduli space of framed $A_Y$-representations
with dimension vector $(n, m+n)$. 
The proof of Theorem~\ref{thm:flop} also 
works in this situation.
Using Example~\ref{exii} (ii), 
the first identity of Theorem~\ref{thm:flop}
becomes
\begin{align*}
\sum_{n, m}A_{n, m[C]}q^n t^{m} =
\prod_{k\ge 1}(1-(-q)^k)^{-2k} \prod_{k\ge 1}(1-(-q)^k t)^k
\prod_{k\ge 1}(1-(-q)^k t^{-1})^k.
\end{align*}
The above formula was first
conjectured by Szendr{\H o}i~\cite{Sz}, 
and later proved by Young~\cite{Young1}, Nagao-Nakajima~\cite{NN}. 
\end{exam}
\begin{remark}
In general for a quiver $Q$ with a super potential $W$, 
we are able to define the ncDT theory for $(Q, W)$. 
A mutation of the pair $(Q, W)$ defines another 
quiver with a super potential $(Q^{\dag}, W^{\dag})$.
The relationship between ncDT invariants on 
$(Q, W)$ and $(Q^{\dag}, W^{\dag})$ is 
described in terms of cluster transformations. 
We refer to~\cite{K-S}, \cite{Nagao} for the detail. 
\end{remark}

\section{Bridgeland stability conditions}\label{sec:Br}
\subsection{Definitions}\label{subsec:defi}
We recall the definition of Bridgeland stability conditions on 
a triangulated category $\mathcal{D}$. 
We fix a finitely generated free abelian group $\Gamma$
with a norm $\rVert \ast \lVert$ on $\Gamma_{\mathbb{R}}$
together with a group homomorphism 
$\mathrm{cl} \colon K(\mathcal{D}) \to \Gamma$.
A typical example is that 
$\mathcal{D}=D^b \mathrm{Coh}(X)$ for a smooth projective
variety $X$, $\Gamma$
is the image of the Chern character map 
$\mathrm{ch} \colon K(X) \to H^{\ast}(X, \mathbb{Q})$, and 
$\mathrm{cl}=\mathrm{ch}$.  
By taking the dual of $\mathrm{cl}$, we always 
regard a group homomorhpism $\Gamma \to \mathbb{C}$
as a group homomorphism $K(\mathcal{D}) \to \mathbb{C}$. 
\begin{definition}[\cite{Brs1}]\label{def:Bst}
A \textit{stability condition} on $\mathcal{D}$ is data
$\sigma=(Z, \{\mathcal{P}(\phi)\}_{\phi \in \mathbb{R}})$, 
where $Z \colon \Gamma \to \mathbb{C}$ is a group 
homomorphism (called \textit{central charge}), 
$\mathcal{P}(\phi) \subset \mathcal{D}$ is a full
subcategory (called $\sigma$-\textit{semistable objects with phase} 
$\phi$) satisfying the following conditions: 
\begin{itemize}
\item For $0\neq E \in \mathcal{P}(\phi)$, 
we have $Z(E) \in \mathbb{R}_{>0} \exp(\sqrt{-1} \pi \phi)$. 
\item  For all $\phi \in \mathbb{R}$, we have 
$\mathcal{P}(\phi+1)=\mathcal{P}(\phi)[1]$. 
\item  For $\phi_1>\phi_2$ and $E_i \in \mathcal{P}(\phi_i)$, we have 
$\Hom(E_1, E_2)=0$. 
\item (Harder-Narasimhan property): 
For each $0\neq E \in \mathcal{D}$, there is
 a collection of distinguished triangles 
$E_{i-1} \to E_i \to F_i \to E_{i-1}[1]$,  
$E_N=E, E_0=0$
with $F_i \in \mathcal{P}(\phi_i)$ and  
$\phi_1> \phi_2> \cdots > \phi_N$. 
\end{itemize}
\end{definition}
Another way defining a stability condition is 
to use a t-structure
as follows:
\begin{lemma}[\cite{Brs1}]\label{lem:tst}
Giving a stability condition on 
$\mathcal{D}$ is equivalent to giving 
data $(Z, \mathcal{A})$, where
$Z \colon \Gamma \to \mathbb{C}$ is a group homomorphism, 
$\mathcal{A} \subset \mathcal{D}$
is the heart of a bounded t-structure, satisfying 
\begin{align}\label{Z}
Z(\mathcal{A} \setminus \{0\}) \in \{ r \exp(i\pi \phi) : r>0, 0<\phi \le 1\}
\end{align}
together with the Harder-Narasimhan property:
for any $E \in \mathcal{A}$, there 
exists a filtration 
$0=E_0 \subset E_1 \subset \cdots \subset E_N=E$
such that $F_i=E_i/E_{i-1}$ is 
$Z$-semistable with 
$\mathrm{arg} Z(F_i)>\mathrm{arg} Z(F_{i+1})$ for all $i$. 
Here $E \in \mathcal{A}$ is $Z$-\textit{semistable} 
if for any subobject $0\neq F \subsetneq E$, 
we have $\mathrm{arg}Z(F)<(\le) \mathrm{arg}Z(E)$. 
\end{lemma}
\begin{proof}
The correspondence is as follows: 
given $(Z, \{\mathcal{P}(\phi)\}_{\phi \in \mathbb{R}})$, 
the corresponding heart $\mathcal{A}$
is the extension closure of $\mathcal{P}(\phi)$ for 
$0<\phi \le 1$. 
Conversely given $(Z, \mathcal{A})$, 
the category $\mathcal{P}(\phi)$ is defined 
to be the category of $Z$-semistable 
objects $E \in \mathcal{A}$ with 
$Z(E) \in \mathbb{R}_{>0} \exp(i\pi \phi)$. 
Other $\mathcal{P}(\phi)$ are defined by the rule 
$\mathcal{P}(\phi+1)=\mathcal{P}(\phi)[1]$.
\end{proof} 
\begin{exam}
Let $C$ be a smooth projective curve, 
$\mathcal{D}=D^b \mathrm{Coh}(C)$,
$\Gamma=\mathbb{Z}^{\oplus 2}$
and $\mathrm{cl}=(\mathrm{rank}, \mathrm{deg})$. 
We set $Z \colon \Gamma \to \mathbb{C}$
to be $(r, d) \mapsto -d+ ir$. 
Then $(Z, \mathrm{Coh}(C))$ is a stability condition, 
whose $Z$-semistable objects coincide with 
classical semistable sheaves on $C$. 
\end{exam}
\subsection{The space of stability conditions}
The space of stability conditions is defined as follows: 
\begin{definition}
We define $\mathrm{Stab}_{\Gamma}(\mathcal{D})$
to be the set of stability conditions
on $\mathcal{D}$ satisfying the 
\textit{support property}, i.e.  
there is a constant $C>0$
such that $\lVert \mathrm{cl}(E) \rVert/\lvert Z(E) \rvert<C$
holds 
for any $0\neq E \in \cup_{\phi \in \mathbb{R}} \mathcal{P}(\phi)$. 
\end{definition}
The main result of Bridgeland~\cite{Brs1} shows that 
the set $\mathrm{Stab}_{\Gamma}(\mathcal{D})$ has a 
structure of a complex manifold. 
If
$\mathcal{D}=D^b \mathrm{Coh}(X)$ for a smooth projective
variety $X$, $\Gamma=\mathrm{Im}(\mathrm{ch})$
and $\mathrm{cl}=\mathrm{ch}$, 
we set $\mathrm{Stab}(X)=\mathrm{Stab}_{\Gamma}(\mathcal{D})$. 
Let $\mathrm{Auteq}(X)$ be the group of exact autequivalences
of $D^b \mathrm{Coh}(X)$. 
The space
$\mathrm{Stab}(X)$ admits
a left action of $\mathrm{Auteq}(X)$
and a right action of $\mathbb{C}$. 
The latter action is given by 
$(Z, \{\mathcal{P}(\phi)\}_{\phi \in \mathbb{R}}) \cdot \lambda =
(e^{-i\pi \lambda}Z, \{\mathcal{P}(\phi + \mathrm{Re}\lambda)\}_{\phi \in \mathbb{R}})$
for $\lambda \in \mathbb{C}$. 
We are interested in the double quotient stack
\begin{align}\label{double}
\left[ \mathrm{Auteq}(X) \backslash
\mathrm{Stab}(X) / \mathbb{C} \right]. 
\end{align}
The conjecture by Bridgeland~\cite{Brs6} is that
if $X$ is a Calabi-Yau manifold, the 
above double quotient stack
contains the stringy K\"ahler moduli space of 
$X$, that is the moduli space of complex structures of 
a mirror manifold of $X$. 

\begin{exam}
(i) If $C$ is an elliptic curve, then 
(\ref{double}) is shown in~\cite{Brs1} to be
isomorphic to 
the modular curve $\mathrm{SL}_2(\mathbb{Z}) \backslash \mathbb{H}$.  
This is compatible with the fact that $C$ is self mirror. 

(ii) Let $\pi \colon X \to \mathbb{P}^2$ be 
the total space of $\omega_{\mathbb{P}^2}$, 
which is a non-compact Calabi-Yau 3-fold. 
In this case, $\mathrm{Stab}(X)$ is defined
to be $\mathrm{Stab}_{\Gamma}(\mathcal{D})$
where $\mathcal{D}$ is the bounded
derived category of compact supported coherent sheaves on $X$, 
$\Gamma$ is the image of $\mathrm{ch} \circ \pi_{\ast}$ in 
$H^{\ast}(\mathbb{P}^2, \mathbb{Q})$, 
and 
$\mathrm{cl}=\mathrm{ch} \circ \pi_{\ast}$. 
Then the quotient stack (\ref{double})
contains $[(\mathbb{C} \setminus \mu_3)/\mu_3]$
by~\cite{BaMa}. 
The latter stack is the parameter space $\psi^3$
of the mirror 
family of $X$ given by 
\begin{align*}
\{y_0^3 + y_1^3 + y_2^3 -3\psi y_1 y_2 y_3=0\}
\subset \mathbb{P}^2. 
\end{align*}
\end{exam}
The space (\ref{double})
is most interesting for projective 
Calabi-Yau 3-folds, e.g. quintic
hypersurfaces in $\mathbb{P}^4$. 
Even in the quintic 3-fold case,  
the space (\ref{double}) 
is very difficult to study. 
In this case, Bridgeland's conjecture~\cite{Brs6}
is stated in the following way: 
\begin{conjecture}\label{conj:quintic}
Let $X \subset \mathbb{P}^4$
be a smooth quintic 3-fold, 
and set $\mathcal{M}_K=[(\mathbb{C} \setminus \mu_5)/\mu_5]$.
Then there is an embedding
\begin{align*}
\mathcal{M}_K \hookrightarrow \left[ \mathrm{Auteq}(X) \backslash
\mathrm{Stab}(X) / \mathbb{C} \right]. 
\end{align*}
\end{conjecture}
The above embedding should be 
given by the solutions of the Picard-Fuchs equation 
which the period integrals
of the mirror family of $X$ satisfy. 
Its explicit description is
available in~\cite{TGep2}.
However in the projective Calabi-Yau 3-fold case, 
it is even not known that
whether $\mathrm{Stab}(X)$ is non-empty or not. 
The first issue in solving Conjecture~\ref{conj:quintic}
is to construct stability conditions, which we discuss in the next
subsection. 
\subsection{Existence problem}
It has been a serious issue to construct
Bridgeland stability conditions on projective 
Calabi-Yau 3-folds. 
Contrary to the one dimensional case, it turns out 
that there is no stability condition 
$\sigma \in \mathrm{Stab}(X)$ of the form 
$\sigma=(Z, \mathrm{Coh}(X))$ if 
$\dim X \ge 2$. From the arguments in string theory, we 
expect the following conjecture:
\begin{conjecture}\label{conj:Bw}
Let $X$ be a smooth projective variety and 
take $B+i\omega \in H^2(X, \mathbb{C})$
with $\omega$ ample class. 
Then there exists the heart of a bounded 
t-structure $\mathcal{A}_{B, \omega}$ on $D^b \mathrm{Coh}(X)$
such that the pair 
$\sigma_{B, \omega}=(Z_{B, \omega}, \mathcal{A}_{B, \omega})$
determines a point in $\mathrm{Stab}(X)$, where $Z_{B, \omega}$ is given by
\begin{align*}
Z_{B, \omega}(E)=-\int_{X} e^{-i\omega} \mathrm{ch}^B(E). 
\end{align*}
Here $\mathrm{ch}^B(E)$ is defined to be $e^{-B} \mathrm{ch}(E)$. 
\end{conjecture}
The resulting stability conditions are expected to form 
a \textit{neighborhood at the large volume limit} in terms of string theory. 
The above conjecture is known to hold if $\dim X \le 2$. 
In the $\dim X=2$ case, the heart $\mathcal{A}_{B, \omega}$
is constructed to be a certain tilting
of $\mathrm{Coh}(X)$, which we are going to review.

Let $X$ be a $d$-dimensional smooth projective variety
and take $B+i\omega \in H^2(X, \mathbb{C})$ with $\omega$
ample. 
The $\omega$-slope function on $\mathrm{Coh}(X)$
is defined to be 
\begin{align*}
\mu_{\omega}(E)=
\frac{\mathrm{ch}_1(E) \cdot \omega^{d-1}}{\mathrm{rank}(E)}
\in \mathbb{R} \cup \{\infty\}.
\end{align*}
Here
$\mu_{\omega}(E)=\infty$ if $\mathrm{rank}(E)=0$. 
\begin{definition}
An object
$E \in \mathrm{Coh}(X)$ is 
$\mu_{\omega}$\textit{-(semi)stable} if for any 
non-zero subobject $0\neq F \subsetneq E$, we have
$\mu_{\omega}(F)<(\le) \mu_{\omega}(E/F)$. 
\end{definition}
We define the pair of subcategories 
$(\mathcal{T}_{B, \omega}, \mathcal{F}_{B, \omega})$
of $\mathrm{Coh}(X)$ to be
\begin{align*}
&\mathcal{T}_{B, \omega}
=\langle E \in \mathrm{Coh}(X): 
E \mbox{ is } \mu_{\omega} \mbox{-semistable with }
\mu_{\omega}(E)> B \omega^{d-1} \rangle \\
&\mathcal{F}_{B, \omega}
=\langle E \in \mathrm{Coh}(X): 
E \mbox{ is } \mu_{\omega} \mbox{-semistable with }
\mu_{\omega}(E)\le B \omega^{d-1} \rangle. 
\end{align*}
Here $\langle \ast \rangle$ means the extension closure. 
The existence of Harder-Narasimhan filtrations 
with respect to the $\mu_{\omega}$-stability 
implies that the pair $(\mathcal{T}_{B, \omega}, \mathcal{F}_{B, \omega})$
is a torsion pair (cf.~\cite{HRS}) of $\mathrm{Coh}(X)$. 
Its tilting defines another heart
\begin{align*}
\mathcal{B}_{B, \omega} =
\langle \mathcal{F}_{B, \omega}[1], \mathcal{T}_{B, \omega}
\rangle \subset D^b \mathrm{Coh}(X). 
\end{align*}
The following result is due to~\cite{Brs2}, \cite{AB}, \cite{Todext}. 
\begin{proposition}
If $\dim X=2$, then 
$(Z_{B, \omega}, \mathcal{B}_{B, \omega}) \in \mathrm{Stab}(X)$. 
\end{proposition}
\begin{proof}
Here is a rough sketch of the proof: 
if $\dim X=2$, then $Z_{B, \omega}(E)$ is written as
\begin{align*}
Z_{B, \omega}(E)=-\mathrm{ch}_2^{B}(E) + \mathrm{ch}_0^B(E)\omega^2/2 +
i\mathrm{ch}_1^{B}(E) \omega. 
\end{align*}
The construction of $\mathcal{B}_{B, \omega}$ immediately 
implies $\mathrm{Im}Z_{B, \omega}(E) \ge 0$
for any $0\neq E \in \mathcal{B}_{B, \omega}$. 
We need to check that 
$\mathrm{Im}Z_{B, \omega}(E)=0$ implies 
$\mathrm{Re}Z_{B, \omega}(E)<0$. 
This property can be easily deduced from the 
the classical Bogomolov-Gieseker (BG) inequality
in Theorem~\ref{thm:BG} below. 
\end{proof}
The following BG inequality played an important role: 
\begin{theorem}[\cite{Bog}, \cite{Gie}]\label{thm:BG}
Let $X$ be a $d$-dimensional 
smooth projective variety, 
and $E$ a torsion free $\mu_{\omega}$-semistable 
sheaf on $X$. Then we have the following inequality:
\begin{align*}
\left(\mathrm{ch}_1^B(E)^2 -2\mathrm{ch}_0^B(E) 
\mathrm{ch}_2^B(E)  \right) \cdot \omega^{d-2} \ge 0. 
\end{align*} 
\end{theorem}

\subsection{Double tilting construction for 3-folds}
Suppose that $X$ is a smooth projective 3-fold, and 
$B$, $\omega$ are defined over $\mathbb{Q}$. 
In this case, the central charge $Z_{B, \omega}$ is written as
\begin{align*}
Z_{B, \omega}(E)=-\mathrm{ch}_3^B(E) + \mathrm{ch}_1^B(E) \omega^2/2
+ i\left(\mathrm{ch}_2^B(E) \omega - \mathrm{ch}_0^B(E) \omega^3/6  \right). 
\end{align*}
Contrary to the surface case, the heart $\mathcal{B}_{B, \omega}$
does not fit into a stability condition with central charge $Z_{B, \omega}$. 
In~\cite{BMT}, Bayer, Macri and the author constructed
a further tilting of $\mathcal{B}_{B, \omega}$
in order to give a candidate of $\mathcal{A}_{B, \omega}$ in 
 Conjecture~\ref{conj:Bw}.
The key observation is the following 
lemma, which also relies on Theorem~\ref{thm:BG}:
\begin{lemma}[\cite{BMT}]
For any $0\neq E \in \mathcal{B}_{B, \omega}$, one of the 
following conditions hold: 

(i) 
$\mathrm{ch}_1^B(E)\omega^2 > 0$. 

(ii) 
$\mathrm{ch}_1^B(E)\omega^2 =0$ and 
$\mathrm{Im}Z_{B, \omega}(E)> 0$. 

(iii) 
$\mathrm{ch}_1^B(E)\omega^2=\mathrm{Im}Z_{B, \omega}(E)=0$
and $\mathrm{Re}Z_{B, \omega}(E)<0$. 
\end{lemma}
The above lemma indicates that the vector 
$(\mathrm{ch}_1^B(E)\omega^2, \mathrm{Im}Z_{B, \omega}(E), 
-\mathrm{Re}Z_{B, \omega}(E))$
behaves as if it were $(\mathrm{rank}, c_1, \mathrm{ch}_2)$
on coherent sheaves on surfaces. 
In~\cite{BMT}, this observation 
led to the following slope function on $\mathcal{B}_{B, \omega}$:
\begin{align*}
\nu_{B, \omega}(E)=\frac{\mathrm{Im}Z_{B, \omega}(E)}{\mathrm{ch}_1^B(E)\omega^2} \in \mathbb{Q} \cup \{\infty\}. 
\end{align*}
Here $\nu_{B, \omega}(E)=\infty$ if $\mathrm{ch}_1^B(E) \omega^2=0$. 
The above lemma shows that $\nu_{B, \omega}$
satisfies the weak see-saw property, and 
it defines a slope stability condition on $\mathcal{B}_{B, \omega}$. 
In~\cite{BMT}, it was called \textit{tilt-stability}:
\begin{definition}
An object $E \in \mathcal{B}_{B, \omega}$ is tilt (semi)stable 
if for any subobject $0\neq F \subsetneq E$, we have
$\nu_{B, \omega}(F)<(\le) \nu_{B, \omega}(E/F)$. 
\end{definition}
We can show the existence of Harder-Narasimhan filtrations 
with respect to the tilt stability. 
Similarly to the surface case, the pair of subcategories 
$(\mathcal{T}_{B, \omega}', \mathcal{F}_{B, \omega}')$
of $\mathcal{B}_{B, \omega}$ defined to be
\begin{align*}
&\mathcal{T}_{B, \omega}'
=\langle E \in \mathcal{B}_{B, \omega}: 
E \mbox{ is tilt semistable with }
\nu_{B, \omega}(E)>0 \rangle \\
&\mathcal{F}_{B, \omega}'
=\langle E \in \mathcal{B}_{B, \omega}: 
E \mbox{ is tilt semistable with }
\nu_{B, \omega}(E)\le 0 \rangle
\end{align*}
is a torsion pair. By tilting, we have another heart
\begin{align*}
\mathcal{A}_{B, \omega} = \langle \mathcal{F}_{B, \omega}'[1], 
\mathcal{T}_{B, \omega}' \rangle \subset D^b \mathrm{Coh}(X). 
\end{align*}
By the construction, we have $\mathrm{Im}Z_{B, \omega}(E)\ge 0$
for any $E \in \mathcal{A}_{B, \omega}$. 
In~\cite{BMT}, we proposed the following conjecture: 
\begin{conjecture}[\cite{BMT}]
If $\dim X=3$, we have 
$(Z_{B, \omega}, \mathcal{A}_{B, \omega}) \in \mathrm{Stab}(X)$. 
\end{conjecture}

\subsection{Conjectural BG inequality for 3-folds}
Our double tilting 
construction led to a 
BG type inequality conjecture 
evaluating the third Chern characters
of tilt semistable objects. 
\begin{conjecture}[\cite{BMT}]\label{conj:BMT}
Let $X$ be a smooth projective 3-fold. 
Then for any tilt semistable object 
$E \in \mathcal{B}_{B, \omega}$ with 
$\nu_{B, \omega}(E)=0$, i.e. 
$\mathrm{ch}_2^B(E)\omega=\mathrm{ch}_0^B(E) \omega^3/6$, 
we have the inequality
\begin{align*}
\mathrm{ch}_3^B(E) \le \frac{1}{18} \mathrm{ch}_1^B(E) \omega^2. 
\end{align*}
\end{conjecture}
\begin{remark}
In order to show $(Z_{B, \omega}, \mathcal{A}_{B, \omega})$
satisfies the property (\ref{Z}), 
it is enough to show the weaker inequality 
$\mathrm{ch}_3^B(E) <\mathrm{ch}_1^B(E) \omega^2/2$. 
If this is true, the existence of HN filtrations is proved in~\cite{BMT}, 
while the support property remains open.
The stronger bound in Conjecture~\ref{conj:BMT}
was obtained by the requirement that 
the equality is achieved for tilt semistable objects
with zero discriminant. 
\end{remark}
It seems to be a hard problem to show Conjecture~\ref{conj:BMT}
even in concrete examples. 
So far, it is proved when $X=\mathbb{P}^3$ by Macri~\cite{MaBo}, $X$ is a 
quadric 3-fold by Schmidt~\cite{BSch}, 
and $X$ is a principally polarized abelian 
3-fold with Picard rank one by Maciocia-Piyaratne~\cite{MaPi1}, \cite{MaPi2}. 
Another kind of evidence is that assuming Conjecture~\ref{conj:BMT}
implies some open problems in other research fields. 
  In~\cite{BBMT}, it was proved that 
Conjecture~\ref{conj:BMT} implies (almost) Fujita conjecture 
for 3-folds: for any polarized 3-fold $(X, L)$, 
$K_X + 4L$ is free and $K_X + 6L$ is very ample. 
In~\cite{TodBG}, it was also proved that 
Conjecture~\ref{conj:BMT} implies 
a conjectural relationship between two kinds of 
DT type invariants
inspired by string theory. 
This result will be reviewed in Theorem~\ref{thm:DM}. 
It may be worth pointing out that, 
in both of the above applications, 
assuming a
weaker inequality, say 
$\mathrm{ch}_3^B(E) <\mathrm{ch}_1^B(E) \omega^2/2$,
does not imply anything. 
The stronger evaluation in Conjecture~\ref{conj:BMT} is crucial
for the proofs of the applications. 

\subsection{The space of weak stability conditions}
Although the existence of Bridgeland stability conditions on projective 
Calabi-Yau 3-folds remains open, 
we are able to modify the definition of Bridgeland stability conditions so that the story in Subsection~\ref{subsec:new} works. The notion of \textit{weak stability conditions} in~\cite{Tcurve1} is one of them. 
In the situation of Subsection~\ref{subsec:defi}, 
we further fix a filtration
$0 \subsetneq \Gamma_0 \subsetneq \cdots \subsetneq \Gamma_N=\Gamma$
such that each subquotient $\Gamma_j/\Gamma_{j-1}$ is a free abelian group. 
Instead of considering a group homomorphism $Z \colon \Gamma \to \mathbb{C}$, 
we consider an element
\begin{align}\label{Zweak}
Z=\{Z_i\}_{j=0}^{N} \in \prod_{j=0}^{N}
\Hom(\Gamma_j/\Gamma_{j-1}, \mathbb{C}). 
\end{align}
Given an element (\ref{Zweak}), we set $Z(v) \in \mathbb{C}$
for $v \in \Gamma$ as follows: 
there is a unique $0\le m\le N$
such that $v \in \Gamma_m \setminus \Gamma_{m-1}$, 
where $\Gamma_{-1} =\emptyset$. 
Then $Z(v)$ is defined to be $Z_m([v]) \in \mathbb{C}$
where $[v]$ is the class of $v$ in $\Gamma_m/\Gamma_{m-1}$. 
\begin{definition}
A weak stability condition on $\mathcal{D}$ with 
respect to the filtration $\Gamma_{\bullet}$ is 
data $(Z, \{\mathcal{P}(\phi)\}_{\phi \in \mathbb{R}})$, 
where $Z$ is as in (\ref{Zweak}), 
$\mathcal{P}(\phi) \subset \mathcal{D}$ 
is a full subcategory, satisfying the same 
axiom in Definition~\ref{def:Bst}. 
\end{definition}
Similarly to Lemma~\ref{lem:tst}, giving a weak stability 
condition is equivalent to giving 
$(Z, \mathcal{A})$, where $Z$ is as in (\ref{Zweak}), 
$\mathcal{A} \subset \mathcal{D}$ is the heart of a bounded t-structure, 
satisfying the same conditions in Lemma~\ref{lem:tst}. 
We denote by 
$\mathrm{Stab}_{\Gamma_{\bullet}}(\mathcal{D})$
the set of weak stability conditions on $\mathcal{D}$
with respect to $\Gamma_{\bullet}$ with a support property. 
This set also has a structure of a complex manifold, 
and coincides with $\mathrm{Stab}_{\Gamma}(\mathcal{D})$
if $N=0$, i.e. the filtration $\Gamma_{\bullet}$ is trivial. 
In general, it is easier to show the non-emptiness for the 
space $\mathrm{Stab}_{\Gamma_{\bullet}}(\mathcal{D})$
with a non-trivial filtration $\Gamma_{\bullet}$. 
The result of Theorem~\ref{thm:main} was obtained by the wall-crossing 
formula in the space of weak stability conditions on the 
following triangulated category
 \begin{align*}
\mathcal{D}_X =\langle \mathcal{O}_X, \mathrm{Coh}_{\le 1}(X) \rangle_{\rm{tr}}
\subset D^b \mathrm{Coh}(X). 
\end{align*}
Here $\mathrm{Coh}_{\le 1}(X)$ is the category of one or zero dimensional
sheaves on $X$, and $\langle \ast \rangle_{\rm{tr}}$ is the 
triangulated closure. 
The relevant data is 
\begin{align*}
\Gamma_0=\mathbb{Z} \oplus H_2(X, \mathbb{Z}) \oplus \{0\} 
\subset \Gamma_1=\Gamma=\mathbb{Z} \oplus H_2(X, \mathbb{Z}) \oplus \mathbb{Z}
\end{align*}
with the map $\mathrm{cl}$ given by 
$\mathrm{cl}(E)=(\mathrm{ch}_3(E), \mathrm{ch}_2(E), \mathrm{ch}_0(E))$. 
Here $H_2(X, \mathbb{Q})$ is identified with $H^4(X, \mathbb{Q})$
via Poincar\'e duality. 
The result of Theorem~\ref{thm:main}
is proved along with the wall-crossing 
argument of Subsection~\ref{subsec:new}
with respect to 
the one parameter family of weak stability conditions on $\mathcal{D}_X$
\begin{align*}
\sigma_{\theta}=
(Z_{\omega, \theta}, \mathcal{A}_X) \in \mathrm{Stab}_{\Gamma_{\bullet}}(\mathcal{D}_X), \quad 1/2 \le \theta <1
\end{align*}
from $\theta=1/2$ to $\theta \to 0$. 
Here $\omega$ is an ample divisor on $X$, 
$Z_{\omega, \theta, j}$ are given by 
\begin{align*}
Z_{\omega, \theta, 0} \colon 
\Gamma_0 \ni (n, \beta) \mapsto n-(\omega \cdot \beta)i, \quad 
Z_{\omega, \theta, 1}:
\mathbb{Z} \ni r \mapsto r \exp(i\pi \theta).
\end{align*}
The heart $\mathcal{A}_X \subset \mathcal{D}_X$
is obtained as the extension closure of objects
$\mathcal{O}_X$ and $\mathrm{Coh}_{\le 1}(X)[-1]$. 
We are able to construct DT type invariant
\begin{align*}
\mathrm{DT}_{\sigma_{\theta}}(1, 0, -\beta, -n) \in \mathbb{Q}
\end{align*}
which counts $\sigma_{\theta}$-semistable 
objects $E \in \mathcal{A}_X$
with $\mathrm{ch}(E)=(1, 0, -\beta, -n)$. 
It is shown that 
\begin{align*}
\mathrm{DT}_{\theta \to 1}(1, 0, -\beta, -n)=P_{n, \beta}, \quad 
\mathrm{DT}_{\theta =1/2}(1, 0, -\beta, -n)=L_{n, \beta} 
\end{align*}
where $L_{n, \beta}$ is the invariant in 
Theorem~\ref{thm:main}. 
The wall-crossing formula describes the difference 
between $P_{n, \beta}$ and $L_{n, \beta}$. 
A similar wall-crossing phenomena
also implies the relationship between $I_{n, \beta}$
and $P_{n, \beta}$ in Theorem~\ref{thm:DTPT}. 
Combined them, we obtain the result of Theorem~\ref{thm:main}. 
Some more detail is also available in~\cite{Tsurvey}. 

\section{Further results and conjectures}\label{sec:Con}
\subsection{Multiple cover formula conjecture}
Although Conjecture~\ref{MNOP}
(i) is proved, a stronger version of the 
rationality conjecture remains open. 
It was proposed by
Pandharipande-Thomas~\cite{PT}, and predicts the product 
expansion (called \textit{Gopakumar-Vafa form})
of the generating series of PT invariants:
\begin{align*}
&1+ \sum_{n \in \mathbb{Z}, \beta>0}
P_{n, \beta}q^n t^{\beta} \\
&=
\prod_{\beta>0} \left(\prod_{j=1}^{\infty}(1-(-q)^j t^{\beta})^{jn_{0}^{\beta}}
\prod_{g=1}^{\infty}\prod_{k=0}^{2g-2}(1-(-q)^{g-1}t^{\beta})^{(-1)^{k+g}n_{g}^{\beta}\left(\begin{subarray}{c}
2g-2 \\ k
\end{subarray}\right)}  \right)
\end{align*}
for some $n_{g}^{\beta} \in \mathbb{Z}$. 
Using Theorem~\ref{thm:main} and Theorem~\ref{thm:DTPT}, 
the above strong rationality conjecture is proved in~\cite{Tsurvey}
to be equivalent to the following conjecture: 
\begin{conjecture}[\cite{JS}, \cite{Tsurvey}]\label{conj:mult}
We have the following identity:
\begin{align*}
N_{n, \beta}=\sum_{k\in \mathbb{Z}_{\ge 1}, k|(n, \beta)}
\frac{1}{k^2}N_{1, \beta/k}.
\end{align*}
\end{conjecture}
The invariant $N_{1, \beta}$ is always integer, 
and the above conjecture is stronger then the 
integrality conjecture by Kontsevich-Soibelman~\cite{K-S}.

\subsection{Gepner type stability conditions}
Let $W \in A=\mathbb{C}[x_1, \cdots, x_n]$ be
a homogeneous polynomial of degree $d$. 
By definition, a graded matrix factorization 
consists of data
\begin{align*}
P^0 \stackrel{p^0}{\to} P^1 
\stackrel{p^1}{\to} P^0(d)
\end{align*}
where $P^i$ are graded free $A$-modules of finite rank, 
$p^i$ are homomorphisms of graded $A$-modules,
$P^i \mapsto P^i(1)$ is the shift of the grading,
satisfying $p^1 \circ p^0 = p^0 \circ p^1 = \cdot W$. 
The triangulated category $\mathrm{HMF}(W)$ is defined to be
the homotopy category of graded matrix factorizations of 
$W$. It has a structure of a triangulated category, 
and related to $D^b \mathrm{Coh}(X)$
for the hypersurface $X=(W=0) \subset \mathbb{P}^{n-1}$. 
For instance if $d=n$, there is an equivalence~\cite{Orsin}
\begin{align}\label{Orlov}
\mathrm{HMF}(W) \stackrel{\sim}{\to} D^b \mathrm{Coh}(X). 
\end{align}
As an analogy of Gieseker stability on 
$\mathrm{Coh}(X)$, we expect the existence
of a natural stability condition on $\mathrm{HMF}(W)$. 
Based on the earlier works~\cite{Wal}, \cite{KST1}, 
the following conjecture is proposed in~\cite{TGep}:
\begin{conjecture}\label{conj:Gep}
There is a Bridgeland stability condition
$\sigma_G=(Z_G, \{\mathcal{P}_G(\phi)\}_{\phi \in \mathbb{R}})$
on $\mathrm{HMF}(W)$
whose central charge $Z_G$ is given by 
\begin{align*}
Z_G \left( \bigoplus_{i=1}^{N} A(m_i) \rightleftarrows \bigoplus_{i=1}^{N}
A(n_i) \right)
= \sum_{i=1}^{N} \left( e^{\frac{2\pi m_i \sqrt{-1}}{d}} - 
e^{\frac{2\pi n_i \sqrt{-1}}{d}}
 \right)
\end{align*}
and the set of semistable objects satisfy $\tau \mathcal{P}_G(\phi)
=\mathcal{P}_G(\phi +2/d)$, 
where $\tau$ is the graded shift functor 
$P^{\bullet} \mapsto P^{\bullet}(1)$. 
\end{conjecture}
If $n=d=5$, i.e. $X$ is a quintic 3-fold, 
a stability condition above 
is expected to correspond
to the orbifold point $0 \in \mathcal{M}_K$
in Conjecture~\ref{conj:quintic}
called Gepner point. 
By this reason, a stability condition in 
Conjecture~\ref{conj:Gep}
is called \textit{Gepner type}. 
Some evidence of Conjecture~\ref{conj:Gep}
is available in~\cite{KST1}, \cite{TGep}, \cite{TGep3}. 
Suppose that Conjecture~\ref{conj:Gep}
is true for $n=d=5$. 
Then as an analogy of Fan-Jarvis-Ruan-Witten theory~\cite{FJRW1} in 
GW theory, we may define the DT type invariant 
\begin{align}\label{DTG}
\mathrm{DT}_G(\gamma) \in \mathbb{Q}, \ \gamma \in \mathrm{HH}_0(W)
\end{align}
which counts $\sigma_G$-semistable
 graded matrix factorizations $P^{\bullet}$ with 
$\mathrm{ch}(P^{\bullet})=\gamma$. 
Here $\mathrm{HH}_0(W)$ is the zero-th Hochschild homology of 
$\mathrm{HMF}(W)$,
and $\mathrm{ch}$ is the Chern character map on graded 
matrix factorizations (cf.~\cite{PoVa}).  
Because of the property of $\sigma_G$, the 
invariant (\ref{DTG}) should satisfy 
$\mathrm{DT}_G(\gamma)=\mathrm{DT}_{G}(\tau_{\ast}\gamma)$.  
Under the Orlov equivalence (\ref{Orlov}), 
the equivalence $\tau$ on the LHS corresponds to the
equivalence $\mathrm{ST}_{\mathcal{O}_X} \circ \mathcal{O}_X(1)$
on the RHS, 
where $\mathrm{ST}_{\mathcal{O}_X}$ is the Seidel-Thomas twist~\cite{ST}
associated to $\mathcal{O}_X$. 
Along with the argument in Subsection~\ref{subsec:new}, the 
existence of the invariant (\ref{DTG}) should imply
a hidden symmetry of the 
generating series of the original DT invariants on the quintic hypersurface
$X=(W=0)$ with respect to the equivalence 
$\mathrm{ST}_{\mathcal{O}_X} \circ \mathcal{O}_X(1)$. 

\subsection{S-duality conjecture for DT invariants}
Let us recall the original S-duality conjecture by Vafa-Witten~\cite{VW}. 
It 
predicts the (at least almost) modularity of  
the generating series of Euler characteristics of 
moduli spaces of stable torsion free sheaves on algebraic surfaces 
with a fixed rank and a first Chern class. 
We refer to~\cite{Goinv} for the developments on the S-duality 
conjecture so far. 
Instead of stable torsion free sheaves on algebraic surfaces,
we consider semistable pure two dimensional 
torsion sheaves on Calabi-Yau 3-folds, 
and DT invariants counting them. 
Let $X$ be a smooth projective Calabi-Yau 3-fold, 
$H$ an ample divisor on $X$ and 
fix a divisor class 
$P\in H^2(X, \mathbb{Z})$. 
We consider the following generating series
\begin{align}\label{intro:DT}
\mathrm{DT}_{H}(P)=\sum_{\beta\in H_2(X), n \in \mathbb{Q}}
\mathrm{DT}_{H}(0, P, -\beta, -n-P \cdot c_2(X)/24)q^n t^{\beta}. 
\end{align}
Here each coefficient counts $H$-semistable 
$E \in \mathrm{Coh}(X)$ whose Mukai vector (\textit{not} Chern character)
satisfies
$\mathrm{ch}(E) \cdot \sqrt{\mathrm{td}_X}=(0, P, -\beta, -n)$. 
As a 3-fold version of the S-duality conjecture, we 
expect that the series (\ref{intro:DT}) satisfies 
a modular transformation property of (almost) Jacobi 
forms.
(We
refer to~\cite{EZ} for a basic of Jacobi forms.)
Some computations of the invariants $\mathrm{DT}_H(0, P, -\beta, -n)$
are available in~\cite{GS3}, \cite{GS1}. 
Also the transformation formula of the series (\ref{intro:DT})
under a flop 
is obtained in~\cite{TodS}. 
Let us consider a flop diagram (\ref{small})
with $Y$ projective,
and $\omega$ an ample divisor on $Y$. 
We assume that the exceptional locus $C$, $C^{\dag}$
of $f$, $f^{\dag}$ are 
isomorphic to $\mathbb{P}^1$
with $p=f(C)=f^{\dag}(C^{\dag})$.  
Let 
$l$ be the scheme theoretic length of 
$f^{-1}(p)$ at the generic point of $C$. 
\begin{theorem}[\cite{TodS}]\label{thm:intro}
There exist
$n_j \in \mathbb{Z}_{\ge 1}$ for $1\le j\le l$ 
such that we have the following formula: 
\begin{align*}
&\mathrm{DT}_{f^{\dag \ast} \omega}(\phi_{\ast}P)
=\phi_{\ast}\mathrm{DT}_{f^{\ast}\omega}(P)
 \\
&\hspace{20mm}\cdot 
\prod_{j=1}^{l} \left\{ 
i^{jP \cdot C-1}
\eta(q)^{-1}\vartheta_{1, 1}(q, ((-1)^{\phi_{\ast}P}t)^{jC^{\dag}}) \right\}^{jn_j P \cdot C}. 
\end{align*}
Here $\phi_{\ast}$ is the variable change
$(n, \beta) \mapsto (n, \phi_{\ast}\beta)$, 
$\eta(q)$ is the Dedekind eta function and 
$\vartheta_{1, 1}(q, t)$ is the Jacobi theta function, 
given as follows:
\begin{align}\label{thetaab}
\eta(q)= q^{\frac{1}{24}}\prod_{k\ge 1}(1-q^k), \quad 
\vartheta_{1, 1}(q, t) =
\sum_{k \in \mathbb{Z}}
q^{\frac{1}{2}\left(k+\frac{1}{2} \right)^2}(-t)^{k+\frac{1}{2}}. 
\end{align}
\end{theorem}
Although $f^{\ast}\omega$ is not ample, it is shown that 
the invariants $\mathrm{DT}_{f^{\ast}\omega}(v)$ are well-defined. 
Recall that $\eta(q)$ is a modular form of weight $1/2$, 
$\vartheta_{1, 1}(q, t)$ is a Jacobi form of weight $1/2$ and index $1/2$. 
The result of Theorem~\ref{thm:intro} shows that 
the series (\ref{intro:DT})
transforms under a flop 
by a multiplication of a meromorphic 
Jacobi form, which gives evidence of the S-duality 
conjecture for DT invariants. 

\subsection{Mathematical approach toward OSV conjecture}
In string theory, the OSV conjecture~\cite{OSV}
predicts
a certain approximation
\begin{align}\label{OSV}
\mathcal{Z}_{\rm{BH}} \sim \lvert \mathcal{Z}_{\rm{top}} \rvert^2
\end{align}
where the LHS is the partition function of black hole entropy, 
and the RHS is the partition function of topological string. 
A version of the above conjecture is mathematically
stated as an approximation between the generating series of 
DT invariants counting torsion sheaves on Calabi-Yau 3-folds, and 
the generating series of GW invariants. 
In~\cite{DM}, Denef-Moore proposed a relationship 
among the series (\ref{intro:DT})
and the generating series
of $I_{n, \beta}$, $P_{n, \beta}$
in order to give a derivation of (\ref{OSV}).  
A mathematical refinement of Denef-Moore conjecture 
is stated in~\cite{TodBG}. 
For simplicity, suppose that $\mathrm{Pic}(X)$ is generated
by an ample divisor $H$. 
For $m\in \mathbb{Z}_{>0}$, we define the following 
cut off series
\begin{align*}
I^m(q, t)=\sum_{(\beta, n) \in C(m)}
I_{n, \beta}q^n t^{\beta}, \quad 
P^m(q, t)=\sum_{(\beta, n) \in C(m)}
P_{n, \beta}q^n t^{\beta}.
\end{align*}
Here $C(m)=\{(\beta, n) : 
\beta H<mH^3/2, \lvert n \vert <m^2 H^3/2\}$. 
Moreover, we define the cut off 
generating series of 
D6-anti-D6 brane counting
\begin{align*}
\mathcal{Z}_{\mathrm{D6}-\overline{\mathrm{D6}}}
(q, t, w)
=\sum_{m_2-m_1=m}
&q^{H^3(m_1^3-m_2^3)/6}t^{H^2(m_1^2-m_2^2)/2}
w^{H^3 m^3/6 + Hc_2(X)m/12} \\
&I^m(qw^{-1}, q^{m_2H}tw^{-mH})
P^m(qw^{-1}, q^{-m_1H}t^{-1}w^{-mH}). 
\end{align*}
\begin{conjecture}[\cite{DM}, \cite{TodBG}]\label{DM}
For $m\gg 0$, we have the equality
\begin{align*}
\mathrm{DT}_H(mH)=
\frac{\partial}{\partial w}
\mathcal{Z}_{\mathrm{D6}-\overline{\mathrm{D6}}}
(q, t, w)\rvert_{w=-1}
\end{align*}
modulo terms of $q^n t^{\beta}$
with 
\begin{align*}
-\frac{H^3}{24}m^3 \left( 1-\frac{1}{m}\right)
\le n + \frac{(\beta \cdot H)^2}{2mH^3}. 
\end{align*}
\end{conjecture}
In~\cite{TodBG}, 
we proved the following: 
\begin{theorem}[\cite{TodBG}]\label{thm:DM}
The unweighted version of Conjecture~\ref{DM}
is true if we assume Conjecture~\ref{conj:BMT}.
\end{theorem}
Even if Conjecture~\ref{DM} is  
proved, still the relationship (\ref{OSV})
is not obvious. If we follow the arguments in~\cite{DM}, 
at least we need to prove S-duality conjecture for DT 
invariants in the previous subsection and 
MNOP conjecture. 
Moreover we need to make a mathematical understanding 
of the approximation $\sim$
in (\ref{OSV}). Although the relationship (\ref{OSV}) is 
motivated by string theory, 
it seems to involve deep and interesting mathematics.

\providecommand{\bysame}{\leavevmode\hbox to3em{\hrulefill}\thinspace}
\providecommand{\MR}{\relax\ifhmode\unskip\space\fi MR }
\providecommand{\MRhref}[2]{%
  \href{http://www.ams.org/mathscinet-getitem?mr=#1}{#2}
}
\providecommand{\href}[2]{#2}

\end{document}